\documentclass[a4paper,12pt,oneside]{amsart}

\usepackage{amsmath,amsfonts,amssymb,amsthm,amsopn}

\newtheorem{Thm}{Theorem}    
\newtheorem{Lem}{Lemma}
\newtheorem{Pro}{Proposition}
\newtheorem{Cor}{Corollary}

\theoremstyle{remark}	
\newtheorem{Rem}{Remark}

\newcommand{\N}{\mathbb{N}}        
\newcommand{\mybox}{\, {\scriptstyle \Box}\,}
\newcommand{\mydiam}{\, {\scriptstyle \Diamond}\,}
\newcommand{\VN}{\mathit{VN}}
\newcommand{\C}{\mathbb{C}}

\begin{document}
\title{The centre of the bidual of Fourier algebras (discrete groups)} 
\author{Viktor  Losert}
\address{Institut f\"ur Mathematik, Universit\"at Wien, Strudlhofg.\ 4,
  A 1090 Wien, Austria}
\email{VIKTOR.LOSERT@UNIVIE.AC.AT}
\date{15 April 2016}

\begin{abstract}
For a discrete group $G$ with Fourier algebra $A(G)$, we study the topological
centre $\mathfrak Z_t$ of the bidual $A(G)''$. If $G$ is amenable, then
$\mathfrak Z_t=A(G)$. But if $G$ contains a free group $F_r$ ($r\ge2$), we
show that $\mathfrak Z_t$ is strictly larger than $A(G)$. Furthermore, it
is shown that the subalgebra $A_\# (G)$ of radial functions in $A(G)$ is
Arens regular.
\end{abstract}

\maketitle
\baselineskip=1.3\normalbaselineskip
Let $A$ be a Banach algebra, $A'$ its dual space. For $f\in A',\;a\in A$\,,
the value of the functional will be written as $\langle f,a\rangle$\,. The
(first) {\it Arens product} $\mybox$ on the bidual space $A''$ is defined as
follows\vadjust{\vskip 1mm} (see also \cite{U}\,p.\,370,
\cite{D}\,(2.6.27), \cite{P}\,Def.\,1.4.1):
\\ First, a right action of $A$ on
$A'$ is defined by
$\langle f\cdot a,b\rangle=\langle f,a\,b\rangle\ (a,b\in A\,,f\in A')$.
Then a left action of $A''$ on $A'$ by
$\langle \Psi\cdot f,a\rangle=\langle\Psi,f\cdot a\rangle\ (\Psi\in
A'')$. Finally, $\langle\Phi\mybox\Psi,f\rangle=\langle\Phi,\Psi\cdot
f\rangle\ \ (\Psi,\Phi\in A'')$.
\\
This makes $A''$ into a Banach algebra and the product is
weak*--continuous when the second factor is fixed. The {\it topological
centre} of $A''$ is defined as
$$\mathfrak Z_t(A'')=\{\Phi\in
A'':\ \Psi\mapsto\Phi\mybox\Psi~ \mbox{is weak*--continuous}\}\;.$$
This can also be described by the condition
$\Psi\mybox\Phi=\Psi\mydiam\Phi$ for all $\Psi\in A''$, where
$\mydiam$ denotes the second Arens product (\cite{D}\,Def.\,2.6.19). If
$A$ is commutative, $\mathfrak Z_t(A'')$ coincides with the algebraic
centre of $A''$. The algebra $A$ is called {\it Arens regular}, if
$\mathfrak Z_t(A'')=A''$.
We will always consider $A$ as a subalgebra of $A''$ (via the
canonical embedding) and, if $A_0$ is a subalgebra of $A$\,, then
$A_0''$ will be identified with the corresponding subalgebra of $A''$
(via the bidual mapping of the inclusion, see also \cite{D}\,p.\,251).
\vspace{0,5cm}

Our principal object will be the Fourier algebra $A(G)$ for discrete
groups $G$\,. It consists of the coefficient functions of the left
regular representation $\rho$ of $G$ on $l^2(G)$ \
(\,$(\rho(g)x)(h)=x(g^{-1}h)$ for $g,h\in G\,,\,x\in l^2(G)$\,), see \cite{E}
for further properties. $A(G)$ is a commutative Banach algebra with respect to
pointwise multiplication, $l ^2(G)\subseteq A(G)\subseteq c_0(G)$. If
$G$ is commutative, $A(G)$ is isomorphic to $L^1(\widehat G)$ (with
convolution), where $\widehat G$ denotes the dual group of $G$\,. General
results on $A(G)''$ can be found in \cite{L0}. Questions
about the topological centre $\mathfrak Z_t(A(G)'')$ have already been
investigated in \cite{LL}, \cite{LL1}. In particular, the case where
$G$ is discrete and amenable was settled in \cite{LL} (when $G$ is
commutative, this goes back further to \cite{LL0}, see also \cite{N}).
For the convolution algebras $L^1(G)$ (and related ones) an account of
results on the topological centre has been given in \cite{BIP}.
\begin{Thm}	  
If $G$ is a discrete amenable group, then \;$\mathfrak Z_t (A(G)'') = A(G)$.
\end{Thm}
\noindent
Banach algebras with this property are sometimes called {\it left strongly
Arens\linebreak irregular}.
\begin{proof}
This is [LL]\,Th.\,6.5\,(i). See also [U]\,Cor.\,2.4 for another argument.
\end{proof}
\vspace{0,5cm}

For non--amenable groups, we restrict mainly to the case where $G$ has a
(non--abelian) free subgroup $F_r$ (see also Remark\,7). If $H$ is a
subgroup of $G$\,, then $A(H)$ can be identified with a subalgebra of
$A(G)$ (consisting of the functions supported by $H$). $B_\rho(G)$
consists of those functions on $G$ that are pointwise limits of
bounded nets (in the $A(G)$--norm) of functions from $A(G)$.
Alternatively, the elements of $B_\rho(G)$ are the coefficients of the
unitary representations of $G$ that are weakly contained in the
regular representation $\rho$\,. $B_\rho(G)$ is again a Banach algebra
under pointwise multiplication.

For $g\in F_r$\,, we denote by $|g|$ the length of the reduced word representing
$g$ (for a fixed choice of free generators). A function $x$ on $F_r$
is called {\it radial} \,if $|g|=|g'|$ implies $x(g)=x(g')$. For finite $r$\,,
we denote by $A_\#(F_r)$ the set of radial elements of $A(G)$ \;(similarly,
$B_{\#,\rho}(F_r)$ in $B_\rho(F_r)$\,). $A_\#(F_r)$ (resp.
$B_{\#,\rho}(F_r)$\,) is a closed subalgebra of $A(F_r)$ (resp.
$B_\rho(F_r)$\,). The harmonic analysis of these spaces (and related
ones) has been described in \cite{FP} (they denote the regular
representation by $\lambda$ and write $B_\lambda(G)$ etc.). Our main
object will be to apply the results of \cite{FP} to questions on Arens
products. This can also be seen as a special example of a symmetric
space (when considering a homogeneous tree of degree $2r$ with a fixed
base point, $F_r$ is isomorphic to the group of ``translations'' of
the tree, radial functions are characterized by invariance under the
``rotations'' of the tree). Results on general symmetric pairs have
been given in \cite{M}.

\begin{Thm}	     
For $G=F_r$\,, the free group with $r$ generators, where $r \geq 2$ is
finite, the algebra \,$A_\# (G)$ \,of radial elements of $A(G)$ and
the algebra \,$B_{\#,\rho}(G)$ \,of radial elements of $B_\rho(G)$ are Arens
regular. \newline
Let \
$J = \{\,\Phi \in A_\# (G)'' :\ \Phi\,\mybox A_\# (G) = (0)\,\}$ \;be the
annihilator of $A_\# (G)$ in $A_\# (G)''$. Then $J$ coincides with the
(left or right) annihilator of $A_\# (G)''$ \;and \linebreak
$A_\# (G)''/J \,\cong\, B_{\#,\rho} (G)$.
\end{Thm}

\begin{Thm}	  
If $G$ is a discrete group containing $F_r$ (the free group with $r$
generators), where $r \geq 2$ is finite, then \;
$A_\# (F_r)'' \subseteq \mathfrak Z_t (A (G)'')$.\newline
In particular, $\mathfrak Z_t (A (G)'') \not= A (G)$.
\end{Thm}
\noindent
The proofs of the Theorems 2 and 3 will be given after several
auxiliary results.
\vspace{0,5cm}

We discuss some further notations. $C ^*_\rho(G)$ denotes the $C^*$--algebra
(on $l^2(G)$) generated by the operators
$\rho(g)\ (g\in G),\ \VN(G)$ stands for the von Neumann algebra generated by
$C^*_\rho(G)$. \,$\VN(G)$ can be identified with a subset of $l^2(G)$
(each operator in $\VN(G)$ is uniquely determined by the image of the
unit vector $\delta_e$ concentrated at the identity $e$ of $G$\,).
$B_\rho(G)$ is isomorphic to the dual space of $C ^*_\rho(G),\ \VN(G)$ is
isomorphic to the dual space of $A(G)$. The (left or right) action of
$A(G)$ on its dual $\VN(G)$ is again given by pointwise multiplication.
If $H$ is a subgroup of $G$, $\VN(H)$ is (isomorphic to) a subalgebra
of $\VN(G)$, and if $E_H(x)$ denotes the restriction of $x$ to $H$\,, then
$E_H\!:\VN(G)\rightarrow \VN(H)$ gives a conditional expectation. We have
again subalgebras in $C_{\#,\rho}^*(F_r)$ (resp. $\VN_\#(F_r)$) of the radial
elements in $C_\rho ^*(F_r)$ (resp. $\VN(F_r)$\,). $E(y)=y ^\#$
(defined as in the proof of Proposition\,2 below) gives a conditional
expectation $E\!:\VN(F_r)\rightarrow \VN_\#(F_r)$ (resp.
$E\!:\VN(G)\rightarrow \VN_\#(F_r)$, when $F_r$ is a subgroup of $G$\,), see
\cite{FP}\,Ch.\,3,\,L.\,1.3.

\begin{Lem}	 
Let $A$ be a Banach algebra.
\begin{enumerate}
\item[(i)] $\Phi \in A''$ vanishes on $A'' \cdot A'$ if and only if
\,$\Phi \mybox A'' = (0)$.
\item[(ii)] If $A$ is commutative, then $\Phi \mybox A'' = (0)$ implies
$A'' \mybox \Phi = (0)$. In particular, it follows that
$\Phi \in\mathfrak Z_t (A'')$.
\end{enumerate}
\end{Lem}

\begin{proof}
(i) follows immediately from the definitions of $\mybox$ and
the action of $A''$ on $A'$. For (ii), recall that for $A$
commutative, $A$ is contained in the algebraic centre of $A''$. Hence
$\Phi \mybox A''=(0)$ implies $A\mybox\Phi=(0)$ and then
weak*--continuity of $\mybox$ gives $A''\mybox\Phi=(0)$.
\end{proof}

\noindent
Applied to our case $A=A(G)$ (arguing as in \cite{LL}\,(4) on p.\,25; in the
notation of Lemma\,2, one has $\Phi\cdot y=p(\Phi)\,y$ for $\Phi\in A(G)'',
\;y\in\VN(G)$\,),
the condition $\mathfrak Z_t(A(G)'')=A(G)$ \;implies that \;
$B_\rho(G)\,\VN(G)$
generates a norm--dense subspace in $\VN(G)$ \;(the action of $B_\rho(G)$ on
$\VN(G)$ amounts to pointwise multiplication).

\begin{Pro} 	  
Let $A$ be Banach algebra with subalgebras $A_0\,,A_1$ and assume that
$\{x\,y\!:\,x\in A_0\,,\,y\in A_1\,,\, \| x \|,\| y \|\leq 1\}$ is weakly
relatively compact. Then \linebreak $A''_0\mybox A''_1\,\subseteq\, A$\,.
\end{Pro}

\begin{proof}
If $C=\{x\,y\!:\,x\in A_0\,,\,y\in A_1\,,\, \| x \|,\| y \|\leq 1\}$, then it
follows immediately from the description of the product by limits
(\cite{D}\,(2.6.28),\,p.\,249) that $\Phi\mybox\Psi\in\overline {C}$ for
$\Phi\in A''_0\,,\,\Psi\in A''_1$ with $\|\Phi\|,\|\Psi\|\leq 1$. \ This is a
variation of a well known result (see \cite{P}\,1.4.13).
\end{proof}

\begin{Pro}	   
Let $G,\,F_r$ be as in Theorem 3. For \,$A_0=B_{\#,\rho}(F_r),\;A=B_\rho(G)$,
we have \;$A_0\,A\subseteq l ^2(F_r)\subseteq A(F_r)$ \,and
\;$\{ x\,y\!:\,x\in A_0\,,\,y\in A\,,\,\| x \|,\| y \|\leq 1\}$ is relatively
compact in $l ^2(F_r)$.
\end{Pro}

\begin{proof}
For $t\in F_r$\,, we write
\;$\Xi(t)=\sup\{\,|x(t)|\!:\,x\in A_\# (F_r),\,\| x \|\leq 1\}$ \,(this follows
the notation of \cite{M}\,p.\,41).
Then, using \cite{FP}\,Ch.\,3,\;Rem.\,2.4\,and\,Th.\,2.2(ii), we deduce that
\;$\Xi(t)=\varphi_\frac 12(t)=(1+\frac {r-1}{r}|t|)\,(2r-1)^{-\frac {|t|}{2}}$.
In particular, $\Xi\in l^3(F_r)$. Put \linebreak
$B_n=\{t\in F_r\!:\,|t|=n\}$. For $y\in A(G)$ and $t\in F_r\,,\ n=|t|$, we
define
\linebreak
$y^\#(t)=\dfrac1{|B_n|}\sum\limits_{|t'|=n}y(t')$.\vadjust{\vskip .4mm}
Then
$y^\# \in A_\# (F_r),\ \| y ^\# \| \le \| y \|$ \
(\cite{FP}\,Ch.\,3,\,L.\,1.3), hence $|y^\#(t)|\leq \|y\|\;\Xi(t)$\,.
Since there always exists $y_0\in A(G)$
with \,$\| y_0\|=\| y\|$ \,and \linebreak
$|y(t)|\leq y_0(t)$ \,for all $t\in F_r$\,, it follows
that\vadjust{\vskip .4mm}
\,$\sum\limits_{t\in B_{n}}|y(t)|^2\leq\| y\|\sum\limits_{t\in B_{n}}|y(t)|
\leq\|y\|^2\sum\limits_{t\in B_{n}}\Xi(t)$. Then for $x\in A_\# (F_r)$,
assuming $\| x\|,\| y\|\leq 1$, we conclude that\vadjust{\vskip .2mm}
\linebreak
$\sum\limits_{t\in B_{n}}|x(t)\,y(t)|^2\leq\sum\limits_{t\in B_{n}}\Xi (t)^3$.
Since (by
\cite{E}\,p.\,195), the unit ball of $B_{\rho}(G)$ is the closure of the
unit ball of $A(G)$ with respect to pointwise convergence, this estimate
extends to \,$x\in B_{\#,\rho}(F_r)$\,, $y\in B_\rho(G)$.
Using \,$\Xi\in l^3(F_r)$,
it follows that given $\varepsilon >0$ there exists $n_0$ such that
$\bigl\|\,x\,y\,\big |\bigcup\limits_{n\ge n_{0}}B_n\bigr\|_2<\varepsilon$ \;
for all $x\in B_{\#,\rho}(F_r),\:y\in B_\rho(G)$
with $\| x\|,\| y\|\leq 1$. This implies relative compactness in $l^2(F_r)$.
\end{proof}

\begin{Rem}	   
Alternatively, we could have used the result of Haagerup
(see \cite{FP} Ch.\,8,\,L.\,1.1) that
\;$\sum\limits_{t\in B_{n}}|y(t)|^2\leq(n+1)^2\,\| y\|^2$ \;holds for
$y\in B_\rho(F_r)$ (and there is also a converse giving a
characterization of those elements of $B(F_r)$ that belong to
$B_\rho(F_r)$\,). Then the same type of estimate as above implies that 
$$B_{\# ,\rho}(F_r)\subseteq \bigcap_{\varepsilon>0}
l^{2+\varepsilon}(F_r)\quad
\text{and}\quad
B_{\# ,\rho}(F_r)\,B_\rho(G)\subseteq \bigcap_{\varepsilon>0}
l^{1+\varepsilon}(F_r)\ .$$
\end{Rem}

\begin{Rem}	   
The results of \cite{FP} mentioned above express an analogue of the
Kunze--Stein phenomenon. If $G$ is a semisimple Lie group with finite
centre, one has \;
$A(G)\subseteq\bigcap\limits_{\varepsilon>0}L^{2+\varepsilon}(G)$.
This can never happen for a discrete group $G$
having an infinite amenable subgroup $H$, since by an easy
approximation and continuity argument, $A(G)\subseteq l ^p(G)$ would
imply $B_\rho(H)\subseteq l ^p(G)$, but $B_\rho(H)$ contains the
constant functions on $H$. See Remark\,4 for further generalizations.
\end{Rem}

\begin{Lem}	  
Let $p\!:A(G)''\rightarrow B_\rho(G)$ be the dual of the inclusion
mapping of $C_\rho ^*(G)$ into $\VN(G)$. Then $p$ is a surjective
algebra homomorphism that restricts to the inclusion mapping on $A(G)$.
\end{Lem}\vspace{-2mm}
\noindent Here $G$ might be any locally compact group.

\begin{proof}
This follows easily from the definition of the Arens product. See also
\linebreak \cite{L}\,Prop.\,5.3.
\end{proof}

\begin{proof}[Proof of Theorem 3]
Recall that for a discrete group $G$\,, we have the continuous embedding
\;$l^2(G)\subseteq A(G)$\,. Thus, it follows from the Propositions\;1\,and\,2
that \linebreak
$A_\#(F_r)''\,\mybox\, A(G)''\subseteq A(G)$ \,and
\;$A(G)''\,\mybox\, A_\#(F_r)''\subseteq A(G)$\,. \,Then for
\,$\Phi\!\in A_\#(F_r)'',\linebreak\Psi\in A(G)''$,
we get from Lemma\,2 (using that $B_\rho(G)$ is commutative)
$$\Phi\mybox\Psi=p(\Phi\mybox\Psi)=p(\Phi)\,p(\Psi)=p(\Psi)\,p(\Phi)=
p(\Psi\mybox\Phi)=\Psi\mybox\Phi\;.$$
(In particular, this gives another argument for the inclusion
\;$B_{\#,\rho}(F_r)\,B_\rho(F_r)\subseteq A(F_r)$\,).
\end{proof}
\vskip 5mm
\begin{proof}[Proof of Theorem 2]
For $G=F_r$\,, we see from Theorem\,3 that $A_\#(F_r)''$ is commutative
and the proof gives that $A_\#(F_r)''\mybox A_\#(F_r)''\subseteq
A(F_r)$. Then Lemma\,2 (and using that the algebra $B_{\#,\rho}(F_r)$
is faithful) implies \;$\ker p\cap A_\#(F_r)''=J$\,. \ Similarly, one shows
the result for $B_{\#,\rho}(F_r)''$.
\end{proof}

\begin{Cor}	    
If $G$ is a discrete group containing $F_r$\,, where $r\geq 2$ is
finite, \linebreak
$J\;(=J(F_r)\,)$ is the annihilator of $A_\#(F_r)''$, then $J$ is
contained in the (left and right) annihilator of $A(G)''$. In
particular, the subspace generated by \,$B_\rho(G)\,\VN(G)$ is not
norm--dense in $\VN(G)$.
\end{Cor}\vspace{-2mm}
\noindent This gives counterexamples to the problem in \cite{LL}\,Rem.\,6.6.

\begin{proof}
If \,$p\!:A(G)''\rightarrow B_\rho(G)$ is defined as in Lemma\,2, it
follows again that\linebreak $J\subseteq \ker p$\,, hence $J\mybox A(G)''=(0)$.
Now use Lemma\,1 and the remark after its proof.
\end{proof}

\begin{Cor}	   
Let $G,F_r$ be as in Corollary 1, $E\!:\VN(G)\rightarrow \VN_\#(F_r)$
denotes the conditional expectation. \
Then \ \ $B_{\#,\rho}(F_r)\,\VN(G)\subseteq C_\rho ^*(G)$ \quad
and \\ $E(\,B_\rho(G)\,\VN(G)\,)\subseteq C_\rho ^*(G)$\,.
\end{Cor}

\begin{proof}
By duality, the first inclusion is equivalent to \;
$\ker p\mybox A_\#(F_r)''=(0)$. For the second part, note that $E$ is the
dual mapping of the inclusion of $C_{\#,\rho}^*(F_r)$ into $C_\rho ^*(G)$.
By an easy argument, $\ker p\cap A_\#(F_r)''=\ker p_\#$\,, where $p_\#$
denotes the dual mapping of the inclusion of $C_{\#,\rho}^*(F_r)$
into $\VN_\#(F_r)$. Again by duality, the second inclusion is
equivalent to \;$\ker p_\#\mybox A(G)''=(0)$.
\end{proof}

\begin{Rem}	    
The Arens product on $A_\#(F_r)''$ has a rather simple description:\\
$\Phi\mybox\Psi=p(\Phi)\,p(\Psi)=p_\#(\Phi)\,p_\#(\Psi)$ \;(where $p$ is
defined in Lemma\,2, $p_\#$ in the proof of Corollary\,2). Thus the
product depends only on the corresponding elements of
$B_{\#,\rho}(F_r)$, the algebra $A_\#(F_r)''$ is an extension of
$B_{\#,\rho}(F_r)$
by the trivial algebra $J\;(=\ker p_\#)$, an annihilator extension
(\cite{D}\,Def.\,1.9.4).
\\
It is a natural question to ask if this extension splits, i.e., if
there is an isomorphic copy of $B_{\#,\rho}(F_r)$ in $A_\#(F_r)''$. We have
$C_{\#,\rho}^*(F_r)\cong C([0,\pi])$, 
$\VN_\#(F_r)\cong L^\infty([0,\pi]),\ A_\#(F_r)\cong L ^1([0,\pi]),\
B_{\#,\rho}(F_r)\cong M([0,\pi])$ \;(see Remark\,4 below). Thus the question
amounts to get a
copy of the space of measures $M([0,\pi])$ in $L^1([0,\pi])''$.
\\
Abstractly, this can be achieved by representing $M([0,\pi])$ as an
$l^1$--sum of spaces $L ^1([0,\pi],\mu_i)$, where $(\mu_i)$ denotes a
maximal family of pairwise singular probability measures in $M([0,\pi])$.
Then one can use Hahn--Banach extensions of the functionals $\mu_i$ on
$C([0,\pi])$ to functionals on $L ^\infty([0,\pi])$. Alternatively, one
can use a Borel lifting on $[0,\pi]$ to get such an extension. But it is
known from results of Shelah that the existence of a Borel lifting
cannot be shown in standard set theory (i.e., one needs additional
assumptions like continuum hypothesis). Thus, although a splitting
exists for this extension, there does not seem to be a ``natural"
construction for such a splitting. Of course, there is the related
question to find an isomorphic copy of $B_\rho(G)$ in $A(G)''$.
Another method to construct a splitting (under appropriate assumptions on
$G$) is given in the proof of Proposition\,3 below.
\\
The structure of  $B_{\#,\rho}(F_r)''$ can be described similarly.
Since $B_{\#,\rho}(F_r)\cong C_{\#,\rho}^*(F_r)'$, we can define
\,$p_1\!:B_{\#,\rho}(F_r)''\to B_{\#,\rho}(F_r)$ as the dual of the
canonical embedding of $C_{\#,\rho}^*(F_r)$ into its bidual
$B_{\#,\rho}(F_r)'$. Thus we have an annihilator extension of
$B_{\#,\rho}(F_r)$ by \,$\ker p_1$ \,with a natural splitting.
Furthermore, in the situation of Theorem\,3,
$B_{\#,\rho}(F_r)''\subseteq\mathfrak Z_t (B_\rho(G)'')$\,. 
\end{Rem}

\begin{Rem}	     
$A(G)$ \;($G$ discrete) and its closed subalgebras belong to the class
of Banach algebras investigated in \cite{U}. The same for $B_{\#,\rho}(F_r)$
\,(using Proposition\,2 for complete continuity).
In particular, $A_\#(F_r)$ and $B_{\#,\rho}(F_r)$
($r\geq 2$, finite) satisfy the properties of \cite{U}\,Th.\,3.3,
characterizing Arens regularity, providing a further example for this
theorem. \!\cite{U} mentioned the example $l^1$ with pointwise
multiplication; the Arens product in this case was already investigated in
the classical paper of Arens and in that of Civin and Yood, see
\cite{P}\,p.\,58. For $l^1$ there is a corresponding decomposition of the
bidual, the r\^ole of $B_{\#,\rho}(F_r)$ is taken by $l^1$ and this
immediately gives a natural splitting for $(l^1)''$ (this is the direct
analogue of the case $B_{\#,\rho}(F_r)$\,).
\\
Since $A(G)$ contains all finitely supported
functions (as a dense subset), it is easy to see that the characters
of $A_\#(F_r)$ are just the evaluation functionals at the points of
$F_r$ (where $|t|=|t'|$ implies that $t,t'$ define the same character).
Thus the character space (Gelfand spectrum) $\Phi_{A_{\#}(F_r)}$ of
$A_\#(F_r)$ can be identified with $\N_0=\{0,1,2,\ldots\}$ (discrete
topology). Since $\varphi_{\frac {1} {2}}(t)\rightarrow 0$ for
$|t|\rightarrow\infty$, the closure $\Phi_{A_\#(F_{r})}\cup\{0\}$ of the
spectrum is even norm--compact in the dual $A_\#(F_{r})'$ \;(compare
\cite{U}\,Cor.\,3.2). Since
$B_{\#,\rho}(F_r)\,B_{\#,\rho}(F_r)\subseteq A_\#(F_{r})$, we have
$\Phi_{A_\#(F_r)}=\Phi_{B_{\#,\rho}(F_r)}$\,.
\\[2mm plus 1mm]
More explicitly, the ``spherical Fourier transform" defines an isomorphism
of the $C^*$--algebras $C_{\#,\rho}^*(F_r)$ and $C([0,\pi])$ \
(\cite{FP}\,Ch.\,3,\,Th.\,3.3, see also Ch.\,3,\,Sec.\,IV; in \cite{FP} they
use $z=\frac12+i\,t$ with $0\le t\le \frac{\pi}{\ln (2r-1)}$ for
parametrization, here we use instead $\theta=t\,\ln (2r-1)$\,).
The Plancherel theorem (\cite{FP}\,Ch.\,3,\,Th.\,4.1) implies that
$\VN_\#(F_r)\cong L^\infty([0,\pi])$ \;(as a $W^*$--algebra), \,
$A_\#(F_r)\cong L ^1([0,\pi],m),\
B_{\#,\rho}(F_r)\cong M([0,\pi])$. Multiplication in $A_\#(F_r)$
corresponds to a
generalized convolution of functions (and measures in the case of
$B_{\#,\rho}(F_r)$\,) on $[0,\pi]$.\vspace{.3mm}
\,For simplicity, we restrict to $r=2$. Then the Plancherel measure is given
by\vspace{.5mm} \;$dm=\dfrac6{\pi\,(4+\cot^2\theta)}\,d\theta$ and the
formula for convolution is \
$\delta_{\theta_1}\ast \delta_{\theta_2}=
s(\theta_1,\theta_2,\theta_3)\,dm(\theta_3)$ \;with\vspace{-2mm}
$$s(\theta_1,\theta_2,\theta_3)=\,
\dfrac{\prod\limits_{j=1}^3\ \bigl(4-3\cos^2(\theta_j)\bigr)}
{6\prod\limits_{\epsilon_2,\epsilon_3=\pm1}
\bigl(2-\sqrt 3\, \cos(\theta_1+\epsilon_2\theta_2+\epsilon_3\theta_3)\bigr)}
$$
Here one can see more directly the ``smoothing effect" of multiplication in
$B_{\#,\rho}(F_r)$ and the compactness statement of Proposition\,2 follows
for radial functions from equicontinuity of the densities
$\theta_3\mapsto s(\theta_1,\theta_2,\theta_3)$ \,(with respect to
$\theta_1,\theta_2$).
\\[0mm plus1mm]
The picture changes when considering the radial Fourier-Stieltjes algebra
$B_\#(F_r)$. One has to enlarge the parameter space by two further (complex)
segments. Put
$J_1=[0,\pi]\cup \{\,i\zeta\,,\,\pi+i\zeta:0\le\zeta\le\frac{\ln(2r-1)}2\,\}$.
Then (\cite{FP}\,Ch.\,3,\,L.\,3.2) we get \ $C_\#^*(F_r)\cong C(J_1)$ (as
$C^*$--algebras) and $B_\#(F_r)\cong M(J_1)$. For \;
$\Im(\theta_1+\theta_2)>\frac{\ln(3)}2$ (again restricting to $r=2$)
an additional atomic part
$A\,\delta_{\theta'}$ appears in the formula for convolution, where
$\theta'=\theta_1+\theta_2-i\,\frac{\ln(3)}2 \mod 2\pi$ \;and
(putting $\nu_j=e^{2\,\Im \theta_j}$)
$A= \dfrac{(3\nu_1-1)(3\nu_2-1)(\nu_1\nu_2-3)}
{12(\nu_1-1)(\nu_2-1)(\nu_1\nu_2-1)}$ \ \vspace{.7mm}
(the expression for the absolutely
continuous part remains the same). Then it is easy to see that $B_\#(F_2)$
is not Arens regular.  More generally, $B_\#(F_r)$ and
$B_\#(F_r)/B_{\#,\rho}(F_r)$ are not Arens regular for any $r\ge 2$\,.
\\[3mm]
In \cite{Mu} the notion of Fourier space $A(H)$ is defined for an arbitrary
locally compact hypergroup~$H$ having a left Haar measure. Then $A_\#(F_r)$
coincides with the Fourier space for the coresponding Sawyer-Voit hypergroup
on $\N_0$ (\cite{BH}\,p.\,183). The argument used in the proof of
Proposition\,2 carries over, if $H$ is any discrete commutative hypergroup
satisfying the {\it Kunze-Stein property} of order $p$ for some $p>\frac43$
(\cite{BH}\,Def.\,2.5.5); equivalently, if $A(H)\subseteq l^q$ for some
$q<4$. General duality theory then implies that in this case $A(H)$ is a
Banach algebra under its standard norm. It follows that $A(H)$ is Arens
regular and the Arens product can be described as in Remark\,3.
\\
In the case of free groups formulas as above go back to Letac (see \cite{V}
for further references). In the literature, this is called a dual convolution
structure (arising on the dual object of a hypergroup). Further results
on multipliers have been shown in \cite{HSS}.
In \cite{V}, detailed computations have been given for the (more general)
Cartier hypergroups (\cite{BH}\,p.\,176), depending on real
parameters $a,b\ge2$ (the case of free groups is $b=2,\;a=2r$). For $a+b>4$,
they satisfy the Kunze-Stein property
for all $p<2$, using the formulas on \cite{V}\,p.\,339. Hence this gives
further explicit examples of Arens regular 
Banach algebras $A(H)$ of the type considered in \cite{U}. But observe
that it follows from the product formulas in \cite{V} that for $a,b>2$
the corresponding Fourier-Stieltjes spaces are no longer algebras under
pointwise multiplication.
\\[1mm]
See also \cite{M} for non-discrete examples coming from spherical functions
on Lie groups.
\end{Rem}

\begin{Rem}	     
For $G,F_r$ as in Theorem\,3, it follows that
\;$p\big(\,\mathfrak Z_t(A(G)'')\big)\supseteq B_{\#,\rho}(F_r)$.
This exhibits another
feature in the non--amenable case. In \cite{LL}\,Th.\,6.4 an essential step in
the proof was to show that \,$p\big(\mathfrak Z_t(A(G)'')\big)=A(G)$ for $G$
amenable (where in the non--discrete case $p$ denotes the dual of the
inclusion mapping of $\mathit{UCB}(\widehat G)$ into $\VN(G)$\,).
\end{Rem}

\begin{Pro}	   
Assume that $G$ is a discrete group and that $A(G)$ has an approximate
identity which is bounded in the
multiplier norm. $p$ is defined as in Lemma~2. Then \vspace*{-2mm}
$$p\big(\,\mathfrak Z_t(A(G)'')\,\big)=\{v\in B_\rho(G)\!:
\;L_v\!:A(G)\to A(G) \text{ is weakly compact\,}\}\,.$$
\end{Pro}

\begin{proof}
$L_v$ denotes the multiplication operator $L_v(u)=v\,u$\,. Recall that if $G$
is a discrete group, then $A(G)$ is an ideal in $A(G)''$ \;(this follows from
\cite{P}\,1.4.13, since \linebreak
$L_v$ is compact for  $v\in A(G)$\,). We use the
characterization of \"Ulger \cite{U}\,Th.\,2.2, \linebreak saying that
$\Phi\in\mathfrak Z_t(A(G)'')$ \;holds for discrete $G$ \;iff 
\;$\Phi\mybox A(G)''\subseteq A(G)$ \,and \linebreak
$A(G)''\mybox\Phi\subseteq A(G)$.
\newline
If $\Phi\in A(G)''$, $v=p(\Phi)$, then for $u\in A(G)$, we get as above \;
$u\mybox\Phi=u\,p(\Phi)=u\,v=L_v(u)$\;. The Arens product being
weak*--continuous in the first variable, it follows that
$\Psi\mybox\Phi=L_v''(\Psi)$ for all $\Psi\in A(G)''$. If
$\Phi\in\mathfrak Z_t(A(G)'')$, then \"Ulger's characterization implies
that $L_v''(A(G)'')\subseteq A(G)$, hence $L_v$ is weakly compact.
\newline
For the converse, let $(u_{\alpha})$ be an approximate identity for $A(G)$
which is bounded in the space of multipliers $M_{\ell}(A(G))$ (notation of
\cite{D}\,Def.\,1.4.25). Put $j(u)=L_u$\,, thus $j\!:A(G)\to M_{\ell}(A(G))$,
and let $\bar\Phi\in M_{\ell}(A(G))''$ be a weak*--cluster point of the
bounded net $(j(u_{\alpha}))$.
\newline
Since the finitely supported functions belong to $A(G)$, the elements
of $M_{\ell}(A(G))$ and $M_{\ell}(B_{\rho}(G))$ are given by functions on
$G$\,. Recall that by
\cite{E}\,p.\,195, the unit ball of $B_{\rho}(G)$ is the closure of the
unit ball of $A(G)$ with respect to pointwise convergence. It follows
easily that $M_{\ell}(B_{\rho}(G))\cong M_{\ell}(A(G))$ (the spaces coincide
when identifying multipliers with functions on $G$). Then by
\cite{D}\,Th.\,2.6.15\,(ii), $B_{\rho}(G)''$ becomes a
left $M_{\ell}(A(G))''$--module (the action will be denoted by $\cdot$\,) with
similar continuity properties as for the Arens product. Since $A(G)$ is an
ideal in $B_{\rho}(G)$, we have $L_u\cdot\Phi\in A(G)''$ for
$u\in A(G),\;\Phi\in B_{\rho}(G)''$ and then $\Psi\cdot\Phi\in A(G)''$ if
$\Psi$~belongs to the weak*--closure of $j(A(G))$ in  $M_{\ell}(A(G))''$.
\newline
For $v\in B_{\rho}(G)$, we put $\Phi_v=\bar\Phi\cdot v$\,. Then
$\Phi_v\in A(G)''$. For $v\in A(G)$, we have
$\bar\Phi\cdot v=\lim u_{\alpha}\,v=v$\,, i.e., $\Phi_v=v$\,. Since for
\,$v\in B_{\rho}(G),\:u\in A(G)$ this implies
$\Phi_v\mybox u = \bar\Phi\cdot(vu)=v\,u$\,, we conclude that \,$p(\Phi_v)=v$
for all $v\in B_{\rho}(G)$. In particular, $v\mapsto \Phi_v$ gives a splitting
for the extension $A(G)''$ of $B_{\rho}(G)$ by $\ker p$ (using the continuity
properties of the Arens product, it is not hard to see that $v\mapsto \Phi_v$
is multiplicative).
\newline
If $L_v$ is weakly compact, then as above \,$v\mybox\Psi=L_v''(\Psi)\in A(G)$
for all $\Psi\in A(G)''$. This implies that \,
$\Phi_v\mybox\Psi=(\bar\Phi\cdot v)\mybox\Psi=\bar\Phi\cdot(v\mybox \Psi)=
v\mybox\Psi\in A(G)$. As at the beginning, we have also
\,$\Psi\mybox\Phi_v=L_v''(\Psi)\in A(G)$ and this leads to
\,$\Phi_v\in \mathfrak Z_t(A(G)''\,)$.\vspace{4mm}
\end{proof}

\noindent
We add a characterization for the weakly compact multipliers of $A(G)$.
\newpage
\begin{Pro}	   
Let $G$ be a discrete group, $m\!:G\to\C$ a function. Then the following
statements are equivalent.
\item[(i)] \,$m$ defines a weakly compact multiplier of $A(G)$.
\item[(ii)] \ $m\,\VN(G)\subseteq C_\rho^*(G)$.
\end{Pro}
\begin{proof}
As in Proposition\,3, we write $L_m(u)=m\,u$\,. Assume that
$L_m\mspace{-4mu}:A(G)\mspace{-1mu}\to A(G)$ is weakly compact.
It is easy to see that the dual mapping $L_m'$ is again given by
$L_m'(y)=m\,y$ \,for $y\in\VN(G)$ (pointwise multiplication). It follows
that $C_\rho^*(G)$ is kept invariant (since the space of finitely supported
functions is invariant). $L_m'$ and its restriction to $C_\rho^*(G)$ are
again
weakly compact. On a weakly compact subset, the weak* topology coincides with
the weak topology. Since the unit ball of $C_\rho^*(G)$ is weak*--dense in
that of $\VN(G)$ and $L_m'$ is weak*--continuous, it follows that $L_m'$
maps $\VN(G)$ to $C_\rho^*(G)$.
\\
Conversely, if $m\,\VN(G)\subseteq C_\rho^*(G)$, then, using the closed graph
theorem, $l_m(y)=m\,y$ \,defines a bounded linear operator
\,$l_m\!:\VN(G)\to C_\rho^*(G)$. We claim that $l_m$ is weakly compact.
Assume first that $G$  is countable. Then $C_\rho^*(G)$ is separable and
weak compactness follows from Pfitzner's result (\cite{Pf}\,Cor.\,7) that
$\VN(G)$ is a Grothendieck space. For general $G$\,, one can use
Eberlein's theorem, noting that any sequence in $\VN(G)$ is contained in
$\VN(H)$ for some countable subgroup $H$ of~$G$\,.
\\
Now, arguing as in the first part, the dual mapping $l_m'$ maps $B_{\rho}(G)$
to $A(G)$ and $l_m'(u)=m\,u$ for $u\in B_{\rho}(G)$. Since $l_m'$ is
weakly compact as well, it follows that $L_m$ is weakly compact.
\end{proof}

\begin{Rem}	     
By results of Haagerup, $G=F_r$ (even when $r$ is infinite) satisfies the
assumptions of Proposition\,3 (see e.g. \cite{FP}\,Ch.\,8,\,Cor.\,1.6).
For non-amenable $G$\,, there
can exist weakly compact multipliers of $A(G)$ which do not belong to
$B_{\rho}(G)$\,: \;e.g. for $G=F_r$ ($r\ge 2$ finite), the proof of
\cite{FP}\,Ch.\,8,\,Prop.\,1.2 shows that the spherical functions
$\varphi_{\sigma}$ have
this property when $\frac12<\sigma<1$. In fact, using this type of estimates
(see also Remark\,1), it follows that every element of
$B_\#(F_r)\cap c_0(F_r)$ defines a compact multiplier of $A(G)$. Another class
of examples is obtained as follows. A subset $E$ of $F_r$ is called a
Leinert set if $l^2(E)\subseteq C^*_\rho(G)$ \ (e.g., if $E$ is a free
subset of $F_r$\,, see \cite{FP}\,Ch.\,2,\,Cor.\,1.4). Then every
$v\in l^{\infty}(E)$ (extended to be zero outside $E$) defines a weakly
compact multiplier of $A(G)$ and this multiplier is compact iff $v\in c_0(E)$,
furthermore \,$l^{\infty}(E)\cap B_{\rho}(F_r)=l^2(E)$.
\\
Proposition\,4 extends similarly to non-discrete groups (with a similar proof).
By a related argument, one can see that another necessary condition for a
function $m$ to define a weakly compact multiplier of $A(G)$ is
\,$m\,B_{\rho}(G)\subseteq A(G)$, but it is not clear if this might be
sufficient as well (I thank the referee for pointing out a [possibly] wrong
assertion in the first version of this paper).
\\
The construction of $\bar\Phi$ in the proof of Proposition\,3 (and its usage
to get a splitting) is a variation of the construction of a right identity in
$A''$ from a bounded right approximate identity in $A$
(\cite{D}\,Prop.\,2.9.16).
By a classical result (see \cite{D} Th.\,4.5.32), $A(G)$ has a bounded
approximate identity iff $G$ is amenable. In the amenable case, it follows
also that $M_{\ell}(A(G))=B_{\rho}(G)\ (=B(G))$. Proposition\,3 (combined with
left strong Arens irregularity of $A(G)$\,) implies that for $G$ discrete and
amenable, every weakly compact multiplier of $A(G)$ is given by an element
of $A(G)$ \;(hence it is already compact).
\\[0mm plus 2mm]
It follows from the characterization of \"Ulger which we used in the proof
of Proposition\,3
that if $G$ is any discrete group, then $\mathfrak Z_t(A(G)'')$ is an
annihilator extension of the
subalgebra \,$p\big(\mathfrak Z_t(A(G)'')\big)$ of $B_\rho(G)$ by the ideal \,
$\mathfrak Z_t(A(G)'')\cap \ker p$\,. An explicit description of these
two subalgebras appears to be not available for a general discrete group $G$.
The elements of  \,$p\big(\,\mathfrak Z_t(A(G)'')\,\big)$ \,are weakly compact
$A(G)$--multipliers of $B_\rho(G)$ into $A(G)$. Conversely, such a
weakly compact multiplier belongs to $p\big(\,\mathfrak Z_t(A(G)'')\,\big)$
if and only if it is the limit of a bounded net (for the
$A(G)$--norm) of multipliers coming from $A(G)$, converging in the strong
operator topology
(i.e., pointwise\,--\,\,norm convergence; recall that approximating
a function in $B_\rho(G)$ pointwise by a bounded net in $A(G)$ gives
in general just pointwise\,--\linebreak weak*-convergence for the
corresponding multipliers on $B_\rho(G)$\,). If $H$ is an amena\-ble
subgroup of $G$, then by \cite{LL}\,Th.\,6.4 (or \cite{U}\,Cor.\,2.4),
$p\big(\,\mathfrak Z_t(A(G)'')\,\big)\cap B_\rho(H)\subseteq A(H)$ \,and
\,$\mathfrak Z_t(A(G)'')\cap \ker p \cap A(H)''=(0)$ \;(in particular, $A(G)$
cannot be Arens regular in case that $H$ is infinite, compare
\cite{U}\,Cor.\,3.7). If $H\cong F_r$ is a
free subgroup of $G$, where $r\geq 2$ is finite, then by Theorem\,3 and
Corollary\,1, $p\big(\,\mathfrak Z_t(A(G)'')\,\big)\supseteq B_{\#,\rho}(H)$
\,and \,$\mathfrak Z_t(A(G)'')\cap \ker p\,\supseteq\, A_\#(H)''\cap \ker p\,=
\,\ker p_\#$\,.
\\[1mm plus 2mm]
As in Lemma\,1, one can easily see that $A'' \mybox \Phi = (0)$
holds if and only if $\Phi \in A''$ vanishes on $A' \cdot A$\,.
Since for $G$ discrete, we have $\VN(G)\,A(G)\subseteq C_\rho ^*(G)$
\,(and the subset is clearly dense),
it follows that $\ker p$ is just the right annihilator of $A(G)''$. But,
in general, the left annihilator is strictly smaller (unless
$A(G)''\cdot \VN(G)\ [=B_{\rho}(G)\,\VN(G)\,]\,\,\subseteq C_\rho ^*(G)$
\,and this is equivalent to Arens regularity of $A(G)$, see also
\cite{U}\,Th.\,3.3\,(e)\,). This means that in most cases $A(G)''$ will not
be an annihilator extension of \,$\ker p$\,.\vspace{0mm plus 3mm}
\end{Rem}

\begin{Rem}	     
Theorem\,3 does not cover all non--amenable discrete groups. An account
of constructions producing non-amenable discrete groups having no non-abelian
free subgroups has been given in \cite{Mo}. A rather extreme class are the
``Tarski monsters" constructed by Ol'shanskii \cite{O}\,Th.\,28.1. They have
no proper infinite subgroups and are non-amenable (see \cite{O}\,p.\,415 for
references). In particular, they have no subgroups isomorphic to $F_r$ for
some $r\geq 2$ and also no infinite amenable subgroups. For such groups,
nothing seems to be known on the size of \,
$\mathfrak Z_t(A(G)'')$, one cannot even exclude that $A(G)$ is Arens
regular (see also \cite{U} and \cite{F} for the question of Arens regularity;
in \cite{F}\,p.\,222, a different example of Ol'shanskii is mentioned for
which $A(G)$ is not Arens regular).
\end{Rem}

\end{document}